\documentclass[12pt]{article}
\usepackage{amsmath,amsthm,amsfonts,amssymb}
\usepackage{amsfonts,epsf,amsmath,tikz}
\usepackage{graphicx,latexsym}
\usepackage{color}
\usepackage{float}
\usepackage[ruled,vlined,linesnumbered]{algorithm2e}

\newtheorem{thm}{Theorem}[section]

\newtheorem{cor}[thm]{Corollary}
\newtheorem{lem}[thm]{Lemma}
\newtheorem{prop}[thm]{Proposition}

\newtheorem{prob}[thm]{Problem}

\newcommand{\gt}{{\rm gt}}

\newcommand{\gpack}{{\rm gpack}}

\newcommand{\SM}{{\rm SM}}

\newcommand{\cp}{\,\square\,}
\newcommand{\strp}{\,\boxtimes\,}

\newcommand{\diam}{{\rm diam}}

\begin{document}
\title{Geodesic packing in graphs}

\author{
	Paul Manuel$^{a}$
	\and
	Bo\v{s}tjan Bre\v{s}ar$^{b,c}$
	\and
	Sandi Klav\v zar$^{b,c,d}$
}

\date{}

\maketitle
\vspace{-0.8 cm}
\begin{center}
	$^a$ Department of Information Science, College of Life Sciences, Kuwait University, Kuwait \\
	{\tt pauldmanuel@gmail.com}\\
	\medskip
	
	$^b$ Faculty of Natural Sciences and Mathematics, University of Maribor, Slovenia\\
	{\tt bostjan.bresar@um.si}\\
	\medskip

	$^c$ Institute of Mathematics, Physics and Mechanics, Ljubljana, Slovenia\\
	\medskip

	$^d$ Faculty of Mathematics and Physics, University of Ljubljana, Slovenia\\
	{\tt sandi.klavzar@fmf.uni-lj.si}\\
\end{center}

\begin{abstract}
A geodesic packing of a graph $G$ is a set of vertex-disjoint maximal geodesics. The maximum cardinality of a geodesic packing is the geodesic packing number ${\gpack}(G)$. It is proved that the decision version of the geodesic packing number is NP-complete. We also consider the geodesic transversal number, ${\gt}(G)$, which is the minimum cardinality of a set of vertices that hit all maximal geodesics in $G$.  While $\gt(G)\ge \gpack(G)$ in every graph $G$, the quotient ${\rm gt}(G)/{\rm gpack}(G)$ is investigated. By using the rook's graph, it is proved that there does not exist a constant $C < 3$ such that $\frac{{\rm gt}(G)}{{\rm gpack}(G)}\le C$ would hold for all graphs $G$. If $T$ is a tree, then it is proved that ${\rm gpack}(T) = {\rm gt}(T)$, and a linear algorithm for determining ${\rm gpack}(T)$ is derived. The geodesic packing number is also determined for the strong product of paths.
\end{abstract}

\noindent{\bf Keywords}: geodesic packing; geodesic transversal; computational complexity; rook's graph; diagonal grid

\medskip
\noindent{\bf AMS Subj.\ Class.}: 05C69; 05C12; 05C85

\section{Introduction}

Pairs of covering-packing problems, known also as dual min-max invariant problems~\cite{AzBu16}, are important topics in graph theory and in combinatorics. The max independent set problem and the min vertex cover problem is an appealing example~\cite{CaFe19}. Another well-known example is the max matching problem versus the min edge cover problem~\cite{gal-59}. Examples from combinatorial optimization are the min set cover problem \& the max set packing problem, and the bin covering \& bin packing problem \cite{HaLa09}. In this paper, we identify a new dual min-max pair: the geodesic transversal problem and the geodesic packing problem. The first one  was recently independently investigated in~\cite{MaBr21, PeSe21a}, here we complement these studies by considering the geodesic packing problem. 

A geodesic (i.e., a shortest path) in a graph $G$ is \textit{maximal} if it is not contained (as a subpath) in any other geodesic of $G$. A set $S$ of vertices of $G$ is a \textit{geodesic transversal} of $G$ if every maximal geodesic of $G$ contains at least one vertex of $S$. When $s\in S$ is contained in a maximal geodesic $P$ we say that vertex $s$ {\em hits} or {\em covers} $P$. 
The \textit{geodesic transversal number} of $G$, $\gt(G)$, is the minimum cardinality of a geodesic transversal of $G$.  
A {\em geodesic packing} of a graph $G$ is a set of vertex-disjoint maximal geodesics in $G$. The {\em geodesic packing number}, $\gpack(G)$, of $G$ is the maximum cardinality of a geodesic packing of $G$, and the {\em geodesic packing problem} of $G$ is to determine $\gpack(G)$.  By a {\em $\gpack$-set} of $G$ we mean a geodesic packing of size $\gpack(G)$. 

Let us mention some related concepts. A packing of a graph often means a set of vertex-disjoint (edge-disjoint) isomorphic subgraphs, that is, the $H$-packing problem for an input graph $G$ is to find the largest number of its disjoint subgraphs that are isomorphic to $H$. In particular, the problem has been investigated for different types of paths. For instance, Akiyama and Chv\'{a}tal~\cite{AkCh90} considered the problem from algorithmic point of view when $H$ is a path of fixed length. A survey on efficient algorithms for vertex-disjoint (as well as edge-disjoint) Steiner trees and paths packing problems in planar graphs was given in~\cite{Wagner93}. Dreier et al.~\cite{DrFu19} have studied the complexity of packing edge-disjoint paths where the paths are restricted to lengths  $2$ and $3$. In~\cite{JiXi16} edge-disjoint packing by stars and edge-disjoint packing by cycles were studied. 

In the rest of this section we first recall some notions needed in the rest of the paper. In the next section it is first proved that the geodesic packing problem is NP-complete. After that we investigate the quotient ${\rm gt}(G)/{\rm gpack}(G)$. We first prove that  $\gt(K_n \cp K_n) = n^2 - 2n + 2$ and use this result to demonstrate that there does not exist a constant $C < 3$ such that $\frac{\gt(G)}{\gpack(G)}\le C$ would hold for all graphs $G$. In Section~\ref{sec:trees} we consider the geodesic packing number of trees and prove that for a tree $T$ we have $\gpack(T)=\gt(T)$. A linear algorithm for determining $\gpack(T)$ is also derived. In the subsequent section the geodesic packing number is determined for the strong product of paths, while the paper is concluded with some closing remarks. 

Let $G=(V(G),E(G))$ be a graph. The order of $G$ will be denoted by $n(G)$. A path on consecutive vertices $a_1, a_2 \ldots, a_k$ will be denoted by $a_1a_2\ldots a_k$. If $n$ is a positive integer, then let $[n]=\{1,\ldots, n\}$. The \textit{Cartesian product} $G \cp H$ of graphs $G$ and $H$ is the graph with the vertex set  $V(G)\times V(H)$ and edges $(g,h)(g',h')$, where either $g = g'$ and $hh'\in E(H)$, or $h=h'$ and $gg'\in E(G)$. The {\em strong product} $G \strp H$ is obtained from $G \cp H$ by adding, for every edge $gg' \in E(G)$ and every edge $hh' \in E(H)$, an edge between the vertices $(g, h)$ and $(g', h')$ and another edge between the vertices $(g, h')$ and $(g', h)$.

\section{Preliminary results and NP-completeness}
\label{sec:prelim}

We start by showing NP-completeness of the geodesic packing problem which is formally defined as follows. 

\begin{center}
	\fbox{\parbox{0.96\linewidth}{\noindent
			{\sc Geodesic Packing Problem}\\[.8ex]
			\begin{tabular*}{\textwidth}{rl}
				{\em Input:} & A graph $G$ and a positive integer $k$.\\
				{\em Question:} & Does there exist a set of $k$ vertex-disjoint maximal geodesics in $G$? 
			\end{tabular*}
	}}
\end{center}

For our reduction we use the concept of  induced path partition. Computationally, given a graph $G$ and a positive integer $k$, the {\sc MaxInduced$P_k$Packing Problem} seeks for a maximum number of vertex-disjoint induced paths $P_k$. Saying that a set of vertex-disjoint induced paths on $k$ vertices is an {\em induced $P_k$-packing} of $G$, the problem is thus to maximize the cardinality of an induced $P_k$-packing. By~\cite[Theorem 3.1]{mt07} we know that {\sc MaxInduced$P_3$Packing Problem} is NP-hard on bipartite graphs with maximum degree $3$. 

Let $G$ be a graph with $V(G)=\{x_1,\ldots,x_n\}$. Then the {\em derived graph} $G'$ is defined as follows: $V(G') = V(G)\cup\{x,y,z\}$ and $E(G')=E(G)\cup\{xz,zy\}\cup\{zx_i:\, i \in [n]\}$. Without any possibility of confusion, we denote by $G$ also the subgraph of $G'$ induced by the vertices of the derived graph $G$.

\begin{lem} 
\label{L:NPcomp}
	A set $\Psi$ is an induced $P_3$-packing of $G$ if and only if $\Psi \cup \{(x,z,y)\}$ is a geodesic packing of the derived graph $G'$.
\end{lem}	

\proof
Note that all maximal geodesics in $G'$ are of length $2$. In particular, the path $P: xzy$ is a maximal geodesic, and every induced path $P_3$ in $G$ is a maximal geodesic in $G'$. The statement of the lemma now follows. 
\qed
\medskip

From Lemma~\ref{L:NPcomp} we also infer that $\gpack(G')=1+pack_{ind}^3(G)$, where we denote by $pack_{ind}^k(G)$ the maximum size of an induced $P_k$-packing in $G$. Now, turning back our attention to the decision versions of the problem, it is easy to see that  an instance $(G, k)$ of the {\sc MaxInduced$P_3$Packing Problem}, where $G$ is a bipartite graph with maximum degree $3$, reduces to an instance $(G', k+1)$ of the {\sc Geodesic Packing Problem}.

\begin{thm}
\label{thm:NP-complete}
	{\sc Geodesic Packing Problem} is NP-complete.
\end{thm}

By Theorem~\ref{thm:NP-complete} it is of interest to bound the geodesic packing number and to determine it for specific families of graphs. The following straightforward upper bound is useful. 

\begin{lem}
	\label{LUpperBoundGpack}
	Let $d$ be the length of a shortest maximal geodesic of a graph $G$. Then, $\gpack(G) \leq \lfloor n(G)/(d+1)\rfloor$.
\end{lem}

Given a set of vertex-disjoint maximal geodesics, each geodetic transversal clearly  hits each of the paths by at least one private vertex of the path. This fact in particular implies the following upper bound. 

\begin{lem}
\label{lem:CLowBound}
If $G$ is a graph, then $\gpack(G) \le \gt(G)$. 
\end{lem}

It is clear that $\gpack(P_n) = 1 = \gt(P_n)$ as well as $\gpack(K_{1,n}) = 1 = \gt(K_{1,n})$, hence the bound of Lemma~\ref{lem:CLowBound} is sharp. On the other hand, the value $\gt(G)$ can be arbitrarily bigger than $\gpack(G)$. For instance, $\gpack(K_{n}) = \lfloor\frac{n}{2}\rfloor$ and $\gt(K_{n})=n-1$.  Observe also that in $K_{n,n}$, $n\ge 2$, every maximal geodesic is of length $2$, hence $\gpack(K_{n,n}) = \lfloor \frac{2n}{3}\rfloor$, while on the other hand $\gt(K_{n,n}) = n$. However, we do not know whether the ratio of the two invariants is bounded and pose this as a problem.

\begin{prob}
\label{pr:ratio}
Is there an absolute constant $C$ such that $\frac{\gt(G)}{\gpack(G)}\le C$, for all graphs $G$?
\end{prob}

The example of complete graphs shows that if the constant $C$ in Problem~\ref{pr:ratio} exists, it cannot be smaller than $2$. To show that it actually cannot be smaller than $3$, consider the rook's graphs~\cite{Hoffman64} that can be described as the Cartesian product of two complete graphs or, equivalently, as the line graphs of complete bipartite graphs~\cite{HIK-2011}. 

\begin{prop} 
	\label{prop:gt-2dim-rook}
	If $n\ge 1$, then $\gt(K_n \cp K_n) = n^2 - 2n + 2$. 
\end{prop}

\begin{proof}
Set $R_n=K_n\cp K_n$, and note that vertices of $R_n$ can be presented in the Cartesian $n\times n$ grid such that two vertices are adjacent if and only if they belong to the same row or the same column. 

For $n=1$, the statement is clear, so let $n\ge 2$. Note that maximal geodesics $P$ in $R_n$ are of length $2$ and consist of three vertices, which can be described as follows: $(g,h)\in V(P)$, and there is a vertex $(g',h)\in V(P)$ in the same column as $(g,h)$ and a vertex $(g,h')\in V(P)$ that is in the same row as $(g,h)$. 
Let $S$ be the complement of a (smallest) $\gt$-set of $R_n$. Hence $S$ contains no maximal geodesics as just described. 

First, we prove that $|S|\le 2n-2$. Let $S_i$ be the set of vertices in $S$ that belong to the $i^{\rm th}$ row of $R_n$. Due to symmetry, we may assume that rows are ordered in such a way that $|S_1|\ge \cdots\ge |S_n|$. Note that $|S_1|=1$, implies $|S|\le n$ and we are done. Hence, let $|S_1|\ge 2$. Note that in the column in which there is a vertex of $S_1$ there are no other vertices of $S$, and the same holds for every row $S_i$ having more than one vertex in $S$. Let $k\ge 1$ be the number of rows in which there are at least two vertices in $S$, That is, in $S_i$, $i\in [k]$, we have $|S_i|\ge 2$, but if $|S_{j}|>0$, where $j>k$, then $|S_j|=1$. Let $C$ be the set of columns in which there are vertices from the sets $S_i$, where $i\in [k]$. Note that there are $|C|$ vertices of $S$ in these columns. Since in the remaining columns there are at most $n-k$ vertices from $S$ (because in each of the remaining rows there is at most one vertex in $S$), we altogether get $|S|\le |C|+n-k$. Now, if $|C|=n$, then $|S|=n$ and we are done. Otherwise, $|S|\le |C|+n-k\le (n-1)+(n-1)=2n-2$. To see that $|S|= 2n-2$, take $k=1$ with $|S_1|=n-1$, and add $n-1$ vertices in the last column to $S$. 
\end{proof}
\medskip

In the proof of Proposition~\ref{prop:gt-2dim-rook} we have reduced the search for the minimum geodesic transversal of rook's graphs to its complement. The latter is equivalent to searching for the largest number of 1-entries in a 0-1 matrix of order $n$, such that the matrix does not contain any of the four $2\times 2$ submatrices with three 1-entries. As one of the reviewers pointed out, this is known to be $2n-2$, however, we were not able to find a reference for it (as this reviewer has also failed to find). Say, in~\cite{furedi-1992}, which is one of the seminal papers on forbidden submatrices, the authors consider 0-1 matrices with four 1-entries and only have Corollary 2.4(1) on matrices with three 1-entries. We also add that the case when $2\times 2$ submatrices with four 1-entries are forbidden is (a special case of) the Zarankiewitz's problem~\cite{z-1951} which is a notorious open problem. Interestingly, it was very recently observed in~\cite[Corollary 3.7]{cicerone-2023}, that the latter problem is equivalent to determine the so-called mutual-visibility number~\cite{distefano-2022} of the rook's graphs.

Since all maximal geodesics in $K_n \cp K_n$ are of length $2$, Lemma~\ref{LUpperBoundGpack} implies that $\gpack (K_n \cp K_n) \le \frac{n^2}{3}$. We can thus estimate as follows: 
\begin{align*}
\frac{\gt(K_n \cp K_n)}{\gpack(K_n \cp K_n)} & \ge \frac{3(n^2 - 2n + 2)}{n^2} = 3\left(1 - \frac{2}{n} + \frac{2}{n^2}\right)\,. 
\end{align*}
Letting $n$ to infinity we have shown that in case the constant $C$ from Problem~\ref{pr:ratio} exists, it cannot be smaller than $3$.

In rook's graphs $K_n \cp K_n$, $n\ge 2$, every maximal geodesic is of length $2 = \diam(K_n \cp K_n)$. More generally, a graph $G$ is {\em uniform geodesic} if every maximal geodesic in $G$ is of length ${\rm diam}(G)$. Complete graphs, cycles, and paths are simple additional families of uniform geodesic graphs. The fact that rook's graphs are uniform geodesic generalizes as follows. 

\begin{prop}
	\label{prp:CP-uniform-geodesic-graphs}
If $G_1,\ldots,G_r$, $r\ge 1$, are uniform geodesic graphs, then the product $G_1 \cp \cdots \cp G_r$ is also a uniform geodesic graph. 
\end{prop}

\proof
The result clearly holds for $r=1$. Moreover, by the associativity of the Cartesian product, it suffices to prove the lemma for two factors. Let hence $P$ be an arbitrary maximal geodesic in $G\cp H$. Then the projections $P_G$ and $P_H$ of $P$ on $G$ and on $H$ are geodesics in $G$ and $H$, respectively. If $P_G$ is not maximal in $G$, then $P_G$ can be extended to a longer geodesic in $G$, but then also $P$ can be extended to a longer geodesic  in $G\cp H$, a contradiction. So $P_G$ and $P_H$ are maximal geodesics in $G$ and $H$, respectively. By our assumption this means that the lengths of $P_G$ and $P_H$ are ${\rm diam}(G)$ and ${\rm diam}(H)$, respectively. As the distance function is additive in Cartesian products, it follows that the length of $P$ is ${\rm diam}(G) + {\rm diam}(H) = {\rm diam}(G\cp H)$. 
\qed
\medskip

Proposition~\ref{prp:CP-uniform-geodesic-graphs}, Lemma~\ref{LUpperBoundGpack}, and the fact that the diameter is also additive on Cartesian products, yield the following result. 

\begin{cor}
If $G_1,\ldots,G_r$, $r\ge 1$, are uniform geodesic graphs, then
$$\gpack(G_1\cp \cdots \cp G_r) \le \left\lfloor \frac{n(G_1)\cdots n(G_r)}{\diam(G_1) + \cdots +\diam(G_r) + 1} \right\rfloor\,.$$
\end{cor}

\section{Trees}
\label{sec:trees}

In this section we derive an efficient algorithm to obtain the geodesic packing number of an arbitrary tree. The approach used is in part similar to the approach from~\cite{MaBr21} to determine the geodetic transversal number of a tree. 

In this section we apply the “\textit{smoothing}” operation on vertices of degree $2$, which is formally defined as follows. Let  $xuy$ be a path of length $2$ in $G$ such that the degree of vertex $u$ in $G$ is $2$. Then a new graph $\SM(G)$ is obtained from $G$ by removing the vertex $u$ and adding the edge $xy$. When there exist two adjacent 2-degree vertices in G, this operation is carried out sequentially one after another.
Let further $\SM(G)$ denote the graph obtained from $G$ by smoothing all the vertices of $G$ of degree $2$. In the smoothing operation, a path $xuy$ is replaced by an edge $xy$ when $\deg(u)=2$ and thus $\deg_{\SM(G)}(v) = \deg_{G}(v)$ for every vertex $v$ in $G'$. Since the smoothing operation preserves the degree of vertices, $\SM(G)$ is well-defined, that is, unique up to isomorphism. It was proved in~\cite[Lemma 4.2]{MaBr21} that $\gt(T) = \gt(\SM(T))$ in any tree $T$. We prove a similar result for the packing invariant. 

\begin{lem}
	\label{lem:treesmoothing}
	If $T$ is a tree, then $\gpack(T) = \gpack(\SM(T))$.
\end{lem}

\proof
Note that each maximal geodesic in a tree connects two leaves of the tree. Let $\Psi_T$ be a largest geodesic packing in $T$. Its elements can thus be represented by pairs of leaves that are endvertices of the corresponding geodesics. Note that a maximal geodesic in $\Psi_T$ from which we remove all vertices of degree $2$ becomes a maximal geodesic in $\SM(T)$.  Thus the same pairs of leaves can be used in $\SM(T)$ to represent the maximal geodesics by its end-vertices. We denote by $\SM(\Psi_T)$ the resulting set of maximal geodesics in $\SM(T)$. Since any two geodesics $g_1,g_2\in \Psi_T$ are disjoint, so are also the corresponding geodesics in $\SM(\Psi_T)$. This implies that $\gpack(T)\le \gpack(\SM(T))$. The reversed inequality can be proved in a similar way. Notably, since the maximal geodesics in $\SM(T)$ have two leaves of $\SM(T)$ as its end-vertices, the same two leaves are end-vertices of a maximal geodesic in $T$. It is clear that the resulting maximal geodesics in $T$ are also mutually vertex-disjoint, and thus together form a geodesic packing in $T$ of cardinality $\gpack(\SM(T))$. Thus, $\gpack(T)\ge \gpack(\SM(T))$.
\qed
\medskip

Lemma~\ref{lem:treesmoothing} does not hold for an arbitrary graph $G$. See Fig.~\ref{fig:CounterExSM}, where a graph $G$ is shown for which we have $\gpack(G) = 4$ and $\gpack(\SM(G)) = 3$. Pairs of endvertices of maximal geodesics are marked by distinct colors.  

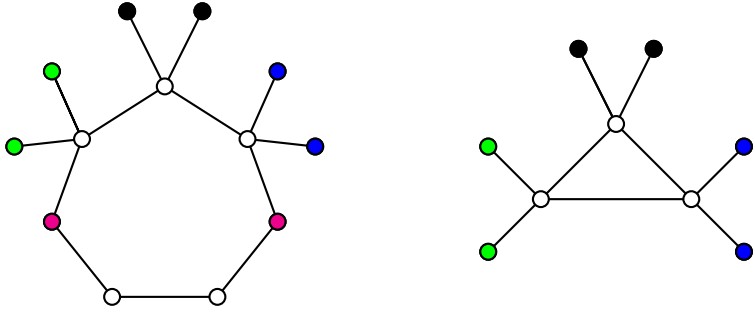
\begin{figure}[ht!]
\begin{center}
\begin{tikzpicture}[scale=1,style=thick,x=1cm,y=1cm]
\def\vr{3pt}

\begin{scope}
\coordinate(1) at (0,0);
\coordinate(2) at (0.5,1);
\coordinate(3) at (0.9,0.1);
\coordinate(4) at (0.5,-1);
\coordinate(5) at (1.3,-2);
\coordinate(6) at (2.7,-2);
\coordinate(7) at (3.5,-1);
\coordinate(8) at (3.1,0.1);
\coordinate(9) at (4,0);
\coordinate(10) at (3.5,1);
\coordinate(11) at (2,0.8);
\coordinate(12) at (2.5,1.8);
\coordinate(13) at (1.5,1.8);
\draw (1) -- (3) -- (2) -- (3); 
\draw (3) -- (4) -- (5) -- (6) -- (7) -- (8) -- (11) -- (3);
\draw (13) -- (11) -- (12);
\draw (10) -- (8) -- (9);
\foreach \i in {1,2,...,13}
{
\draw(\i)[fill=white] circle(\vr);
}
\foreach \i in {1,2}
{
\draw(\i)[fill=green] circle(\vr);
}

\foreach \i in {9,10}
{
\draw(\i)[fill=blue] circle(\vr);
}

\foreach \i in {13,12}
{
\draw(\i)[fill=black] circle(\vr);
}
\foreach \i in {4,7}
{
\draw(\i)[fill=magenta] circle(\vr);
}

\end{scope}
		
\begin{scope}[xshift=6cm, yshift=-0.7cm]
\coordinate(1) at (0.3,-0.7);
\coordinate(2) at (0.3,0.7);
\coordinate(3) at (1,0);
\coordinate(4) at (2,1);
\coordinate(5) at (1.5,2);
\coordinate(6) at (2.5,2);
\coordinate(7) at (3,0);
\coordinate(8) at (3.7,0.7);
\coordinate(9) at (3.7,-0.7);
\draw (1) -- (3) -- (2); 
\draw (3) -- (7) -- (4) -- (5);
\draw (5) -- (4) --(3); \draw (4) --(6); 
\draw (8) -- (7) -- (9);
\foreach \i in {1,2,...,9}
{
\draw(\i)[fill=white] circle(\vr);
}

\foreach \i in {1,2}
\draw(\i)[fill=green] circle(\vr);

\foreach \i in {9,8}
{
\draw(\i)[fill=blue] circle(\vr);
}

\foreach \i in {5,6}
{
\draw(\i)[fill=black] circle(\vr);
}

\end{scope}

\end{tikzpicture}
	\caption{A graph $G$ with ${\gpack}(G) = 4$,  and $\SM(G)$ with ${\gpack}(\SM(G))=3$.}
	\label{fig:CounterExSM}
\end{center}
\end{figure}

A {\em support vertex} in a tree is a vertex adjacent to a leaf.  An {\em end support vertex} is a support vertex that has at most one non-leaf neighbor. It is easy to see that an end support vertex does not lie between two end support vertices. In addition, every tree on at least two vertices contains an end support vertex (see, for instance,~\cite{MaBr21}). In~\cite[Lemma 4.3]{MaBr21} the following result was proved. 

\begin{lem}\hskip 0.1cm {\rm \cite{MaBr21}}
\label{lem:treetransversal}
Let $T$ be a tree with no vertices of degree $2$. Let $u$ be an end support vertex of $T$ and $u_1, \ldots, u_s$ the leaves adjacent to $u$. Then $\gt(T) = \gt(T -\{u, u_1, \ldots, u_s \})+1$. Moreover, there exists a gt-set $S$ of $T$ such that $u\in S$.
\end{lem}

We prove a result parallel to Lemma~\ref{lem:treetransversal} concerning the geodesic packing number. 
\begin{lem}
\label{lem:treepacking}
Let $T$ be a tree with no vertices of degree $2$. Let $u$ be an end support vertex of $T$ and $u_1, \ldots, u_s$ the leaves adjacent to $u$. Then $\gpack(T) = \gpack(T - \{u, u_1, \ldots, u_s \})+1$. Moreover, there exists a $\gpack$-set $\Psi$ of $T$ such that $u_1uu_2\in \Psi$.
\end{lem}

\proof
Since $T$ has no vertices of degree $2$, the end support vertex $u$ is adjacent to at least two leaves, that is, $s\ge 2$. If $T$ is a star, and hence $u$ being the center of it, then the assertion of the lemma is clear. In the rest of the proof we may thus assume that $u$ has at least one non-leaf neighbor,
and since $u$ is an end support vertex, it has only one non-leaf neighbor. We denote the latter vertex by $w$, and let $T'$ be the component of $T-u$ that contains the vertex $w$. 

Let $\Psi'$ be a $\gpack$-set of $T'$. Since $u_1uu_2$ is a maximal geodesic in $T$, and every maximal geodesic in $T'$ is a maximal geodesic also in $T$, we infer that $\Psi'\cup\{u_1uu_2\}$ is a geodesic packing of $T$. Hence $\gpack(T)\ge \gpack(T')+1$.

Note that there can be at most one maximal geodesic in a geodesic packing of $T$ that contains vertex $u$. In addition, there is at least one geodesic that contains $u$ if a geodesic packing of $T$ is of maximum cardinality (for otherwise, one could add the geodesic $u_1uu_2$ and make it of larger cardinality, which is a contradiction). Now, let $\Psi$ be a $\gpack$-set of $T$ and let $P\in \Psi$ be the geodesic that contains $u$. It is easy to see that all maximal geodesics in $\Psi\setminus\{P\}$ belong to $T'$ and are also pairwise vertex-disjoint maximal geodesics of $T'$. Hence $\gpack(T')\ge \gpack(T)-1$, and we are done. 
\qed
\medskip

Combining the facts that $\gpack(K_2)=1=\gt(K_2)$, that in any tree $T$ we have $\gt(T) = \gt(\SM(T))$ and $\gpack(T) = \gpack(\SM(T))$, and using Lemmas~\ref{lem:treetransversal} and \ref{lem:treepacking}, we deduce the following result.

\begin{thm}
\label{thm:treeequality}
If $T$ is a tree, then $\gpack(T)=\gt(T)$.
\end{thm}

Using the lemmas from this section, we can now present an algorithm that constructs a $\gpack$-set of an arbitrary tree $T$. Note that a $\gpack$-set of $T$ is uniquely determined by pairs of endvertices of its maximal geodesics, and the outcome of the algorithm is the set of such (ordered) pairs.

\begin{algorithm}[hbt!]
\label{al:tree}
\caption{ $\gpack$-set of a tree}
\label{alg:gt-set-tree} 
\KwIn{A tree $T$.}
\KwOut{A $\gpack$-set $\Psi$, represented by pairs of end-vertices.}

\BlankLine
{
$\Psi=\emptyset$\\
$T=\SM(T)$ \\
\While{$n(T) \ge 3$}
{
 identify an end support vertex $p$ of $\SM(T)$, and its leaf-neigbors $u_1,u_2$\\
 $\Psi=\Psi\cup\{(u_1,u_2)\}$\\
 $T=T-\{p,u_1,\ldots,u_t\}$, where $u_1,\ldots,u_t$ are the leaf neighbors of $p$\\ 
 $T=\SM(T)$ \\
}
\If{$n(T) = 2$}
{$\Psi=\Psi \cup V(T)$}
}
\end{algorithm}

\begin{thm}
\label{thm:tree}
Given a tree $T$, Algorithm~\ref{al:tree} returns the set of pairs of end vertices of maximal geodesics of a $\gpack$-set of $T$ in linear time.  
\end{thm}

The correctness of Algorithm \ref{alg:gt-set-tree} follows from Lemmas~\ref{lem:treesmoothing} and \ref{lem:treepacking}. The time complexity of the algorithm is clearly linear. For the running time of the algorithm, in Step 7, there is nothing to be done if $T$ is a star. Otherwise, the unique non-leaf neighbor of the vertex $p$ selected in Step 4 is the only vertex for which we need to check whether the smoothing operation is required.

\section{Diagonal grids}
\label{sec:diagonal-grid}

{\em Diagonal grids} are strong products of paths~\cite{HIK-2011}. If a diagonal grid is the strong product of $r$ paths, then it is called an {\em $r$-dimensional diagonal grid}. By definition, the $r$-dimensional grid $P_{d_1} \cp  \cdots \cp P_{d_r}$  is a spanning subgraph of $P_{d_1} \strp \cdots \strp P_{d_r}$, cf.\ Fig.~\ref{fig:GridDiagonalGrid}. The edges of $P_{d_1} \cp \cdots \cp P_{d_r}$ (considered as a subgraph of  $P_{d_1} \strp \cdots \strp P_{d_r}$) are called {\em Cartesian edges} of $P_{d_1} \strp \cdots \strp P_{d_r}$, the other edges are {\em diagonal edges}. We say that a geodesic consisting of only Cartesian edges is a {\em Cartesian geodesic} of $P_{d_1} \strp \cdots \strp P_{d_r}$. In the rest we will assume that the vertices of a path on $r$ vertices are integers $1,\ldots, r$, and if $x\in V(P_{d_1} \strp \cdots \strp P_{d_r})$, then we will use the notation $x = (x_1,\ldots, x_r)$. 
\begin{figure}[ht!]
	\begin{center}
	\begin{tikzpicture}
	\def\vr{3pt}
	\foreach \x in {1,...,5}
	\foreach \y in {1,...,4}
	{
		\draw[thick](\x,\y) +(-.5,-.5)  rectangle ++(.5,.5);
		\draw[thick](\x,\y) +(-.5,-.5)[fill=white] circle(\vr);
	}
	\draw[thick](6,1) +(-.5,-.5)[fill=white] circle(\vr);
	\draw[thick](6,2) +(-.5,-.5)[fill=white] circle(\vr);
	\draw[thick](6,3) +(-.5,-.5)[fill=white] circle(\vr);
	\draw[thick](6,4) +(-.5,-.5)[fill=white] circle(\vr);
	\draw[thick](6,5) +(-.5,-.5)[fill=white] circle(\vr);
	\draw[thick](1,5) +(-.5,-.5)[fill=white] circle(\vr);
	\draw[thick](2,5) +(-.5,-.5)[fill=white] circle(\vr);
	\draw[thick](3,5) +(-.5,-.5)[fill=white] circle(\vr);
	\draw[thick](4,5) +(-.5,-.5)[fill=white] circle(\vr);
	\draw[thick](5,5) +(-.5,-.5)[fill=white] circle(\vr);
\end{tikzpicture}
\begin{tikzpicture}
	\def\vr{3pt}
	\foreach \x in {1,...,5}
	\foreach \y in {1,...,4}
	{
		\draw[thick](\x,\y) +(-.5,-.5)  rectangle ++(.5,.5);
		\draw[thick](\x,\y) +(-.5,-.5) --  ++(.5,.5);
		\draw[thick](\x,\y) +(-.5,.5) --  ++(.5,-.5);
		\draw[thick](\x,\y) +(-.5,-.5)[fill=white] circle(\vr);
	}
	\draw[thick](6,1) +(-.5,-.5)[fill=white] circle(\vr);
	\draw[thick](6,2) +(-.5,-.5)[fill=white] circle(\vr);
	\draw[thick](6,3) +(-.5,-.5)[fill=white] circle(\vr);
	\draw[thick](6,4) +(-.5,-.5)[fill=white] circle(\vr);
	\draw[thick](6,5) +(-.5,-.5)[fill=white] circle(\vr);
	\draw[thick](1,5) +(-.5,-.5)[fill=white] circle(\vr);
	\draw[thick](2,5) +(-.5,-.5)[fill=white] circle(\vr);
	\draw[thick](3,5) +(-.5,-.5)[fill=white] circle(\vr);
	\draw[thick](4,5) +(-.5,-.5)[fill=white] circle(\vr);
	\draw[thick](5,5) +(-.5,-.5)[fill=white] circle(\vr);
\end{tikzpicture}
	\caption{(a) A $2$-dimensional grid $P_6\cp P_5$  and (b) a $2$-dimensional diagonal grid $P_6\strp P_5$}
\label{fig:GridDiagonalGrid}
\end{center}
\end{figure}
\begin{lem}
	\label{LLengthGeoDiaGrid}
	If $P$ is a maximal geodesic in  $P_{d_1} \boxtimes \cdots \boxtimes P_{d_r}$, where $r\ge 2$, and $d_1, \ldots, d_r\ge 2$, then $n(P) \in \{d_1, \ldots, d_r\}$.
\end{lem}

\proof
	Let $P$ be an arbitrary geodesic of $G = P_{d_1} \boxtimes \cdots \boxtimes P_{d_r}$ of length $\ell \ge 2$, so that $n(P) = \ell +1$. Let $xx'$ and $yy'$ be the first and the last edge of $P$, where $x$ and $y'$ are the first and the last vertex of $P$, respectively.  It is possible that $x' = y$. Then $\ell = d_G(x,y') = 1 + d_G(x',y) + 1$. (Note that if $x' = y$, then $d_G(x',y) = 0$.) 
	
	Since $d_G(x,y') = \max \{ |x_1 - y_1'|, \ldots, |x_r - y_r'|\}$, we may without loss of generality assume (having in mind that the strong product operation is commutative) that $\ell = d_G(x,y') = |x_1 - y_1'|$. We now claim that $y_1 \ne y_1'$ and suppose on the contrary that $y_1 = y_1'$. Using the facts that $d_G(x',y) = \max \{ |x_1' - y_1|, \ldots, |x_r' - y_r|\}$, $|x_1 - y_1'| = \ell$, $|x_1 - x_1'| \le 1$, and $y_1 = y_1'$, we get that  $|x_1' - y_1| \ge \ell -1$. Consequently,  $d_G(x',y) \ge \ell -1$, which in turn implies that  
	$$\ell = d_G(x,y') = 1 + d_G(x',y) + 1 \geq 1 + (\ell -1) + 1 = \ell + 1\,,$$
	a contradiction. We have thus proved that if $d_G(x,y') = |x_1 - y_1'|$, then $y_1 \ne y_1'$. Let us emphasize that $P$ was assumed to be an arbitrary geodesic. 
	
	Let now $P$ be a maximal geodesic in $G$ and use the same notation as above. Assume again wlog that $\ell = d_G(x,y') = |x_1 - y_1'|$. If $uv$ is an arbitrary edge of $P$ which is different from $xx'$, then the above claim asserts that $u_1\ne v_1$. Since $\ell = d_G(x,y') = |x_1 - y_1'|$ it follows that the first coordinates of the vertices of $P$ are $\ell + 1$ consecutive integers $i, i+1, \ldots, i+\ell$. If $i > 1$, then adding the edge between $x$ and the vertex $(i-1,x_2, \ldots, x_r)$ yields a geodesic which strictly contains  $P$, a contradiction. Hence $i=1$. By a parallel argument we get that $i+\ell = d_1$. We conclude that $n(P) = d_1$. 
\qed
\medskip

From the proof of Lemma~\ref{LLengthGeoDiaGrid} we can deduce also the following. 

\begin{lem}
\label{lem:explicit-path}
Let $G = P_{d_1} \strp \cdots \strp P_{d_r}$, where $r\ge 2$ and $d_i\ge 2$ for $i\in [r]$. If $x = (x_1, \ldots, x_{i-1}, 1, x_{i+1}, \ldots x_r)$ and $y = (y_1, \ldots, y_{i-1}, d_i, y_{i+1}, \ldots y_r)$ are vertices of $G$ with $d_G(x,y)=d_i-1$, then there exists a maximal $x,y$-geodesic in $G$ of length $d_i - 1$. 
\end{lem}

We are now in position to determine the geodesic packing number of diagonal grids. 

\begin{thm}
	\label{TDiaGridPack}
If $r\ge 2$ and $2\le d_1 \leq \min\{d_2, \ldots, d_r\}$, then 	
$$\gpack(P_{d_1} \strp \cdots \strp P_{d_r}) = d_2\cdot d_3 \cdots d_r\,.$$
\end{thm}

\proof
Set $G = P_{d_1} \strp \cdots \strp P_{d_r}$. For each vector $(i_2,\ldots, i_r)$, where $i_j\in [d_j]$, $j\in \{2,\ldots, r\}$, let $P_{i_2,\ldots, i_r}$ be the path
$$(1,i_2,\ldots, i_r)(2,i_2,\ldots, i_r) \ldots (d_1,i_2,\ldots, i_r)\,.$$
By Lemma~\ref{lem:explicit-path}, $P_{i_2,\ldots, i_r}$ is a maximal geodesic of $G$. Hence the set 
$$\{P_{i_2,\ldots, i_r}:\ i_j\in [d_j], j\in \{2,\ldots, r\}\}$$ 
is a geodesic packing of $G$. Its size is $d_2\cdot d_3 \cdots d_r$ which means that  hence $\gpack(G) \geq d_2\cdot d_3 \cdots d_r$.  

From Lemma~\ref{LLengthGeoDiaGrid} we know that a shortest maximal geodesic of $G$ is of length $d_1-1$. This implies, by using Lemma~\ref{LUpperBoundGpack}, that $\gpack(G) \leq n(G)/ d_1 = d_2\cdot d_3 \cdots d_r$ and we are done.   
\qed

\section{Conclusions}

We have introduced the geodesic packing problem which is a min-max dual  invariant to the earlier studied geodesic transversal problem. We have settled the complexity status of the geodesic packing problem for general graphs and arbitrary trees, and determined the geodesic packing number for several classes of graphs.
We have proved that $\gpack(T)=\gt(T)$ for arbitrary trees $T$. It is not known that $\gpack(G)=\gt(G)$ when $G$ is a cactus graph or block graphs.
There are numerous open problems that are left for future investigation. One open problem is explicitly stated in Problem~\ref{pr:ratio}. Other natural extensions of our research would be to study the geodesic packing number for general strong products or other graph products and the general packing number for intersection graphs such as interval graphs, circular arc graphs or chordal graphs. 
\section*{Acknowledgments}
This work was supported and funded by Kuwait University, Research Project No.\ (FI01/22).
\section*{Conflict of interest}
The authors declare that they have no conflict of interest.

\end{document}